\documentclass[12pt]{amsart}
\usepackage{amsmath}
\usepackage{amsthm}
\usepackage{color}
\usepackage{amscd}
\usepackage{amsfonts}
\usepackage{graphicx}
\usepackage{epsfig}
\usepackage{amssymb,latexsym}
\usepackage{bm}

\newcommand{\dv}{\operatorname{div}}

\def\bbn{{\boldsymbol{n}}}
\def\bbt{{\boldsymbol{t}}}

\def\bu{\mathbf{u}}

\def\bv{\mathbf{v}}
\def\bw{\mathbf{w}}

\def\bbx{{\boldsymbol{x}}}

\def\R{\mathbb{R}}

\def\bQ{\mathbf{Q}}
\def\bbH{\mathbb H}
\def\bbL{\mathbb L}

\def\bB{\mathbf B}

\def\id{\mathbf{id}}
\def\bcero{\mathbf 0}
\newtheorem{theorem}{Theorem}[section]
\newtheorem{corollary}[theorem]{Corollary}
\newtheorem{definition}{Definition}[section]
 
\newtheorem{proposition}[theorem]{Proposition}

\numberwithin{equation}{section}

\begin{document}

\title[Conjugate harmonic functions]{Conjugate harmonic functions in 3D with respect to a unitary gradient}
\author{Pablo Pedregal}
\address{E.T.S. Ingenieros Industriales. Universidad de Castilla La
Mancha. Campus de Ciudad Real, Spain} \email{pablo.pedregal@uclm.es}

\date{} 
\begin{abstract}
We propose to relax the classic Cauchy-Riemann equations for a mapping. We support the interest of such a proposal by looking at one specific situation in 3D, and proving the existence of pairs of harmonic conjugate functions with respect to a unitary gradient as the title of this contribution conveys. We further investigate the relationship between boundary conditions for such pairs, the importance of the unitary constraint, and the eventual link of these ideas to Calderón's problem in 3D.
\end{abstract}
\maketitle
\section{Introduction}
The matrix equation 
\begin{equation}\label{conformal}
\nabla\bu\,\nabla\bu^T=\det\nabla\bu^{2/N}\,\id,
\end{equation}
for a mapping 
$$
\bu(\bbx):\Omega\subset\R^N\to\R^N, \quad \nabla\bu=\begin{pmatrix}\nabla u_1\\\vdots\\\nabla u_N\end{pmatrix}, \quad \bu=(u_1, u_2, \dots, u_N), 
$$ 
is usually known as the Cauchy-Riemann equations for $\bu$, as it is an appropriate generalization of the so well-known conditions for analytic complex functions for the 2D case $N=2$ (see, among many other places, \cite{alhfors}, \cite{astaliwa}, \cite{iwaniec}, \cite{lehto}). $\id$ is the identity matrix of dimension $N$. 
It is also very well-known that the differences between the cases $N=2$ and $N\ge3$ are in some sense dramatic. The first fundamental difference is the celebrated Liouville theorem which ensures that the only such mappings when $N\ge3$ are either constants or (restrictions of) M\"obius transformations (see above references).

The question we would like to raise is whether it would be interesting to relax equation \eqref{conformal} to 
 \begin{equation}\label{nuevaz}
\nabla\bu\,\nabla\bu^T=D(\det\nabla\bu^{2\alpha_i}),
\end{equation}
for a family of exponents $\alpha_i$, $i=1, 2, \dots, N,$ such that 
\begin{equation}\label{exponentes}
\quad \sum_i\alpha_i=1, \alpha_i\ge0,
\end{equation}
and investigate classes of functions that comply with such a condition. We are using the notation $D(\lambda_i)$ to indicate a diagonal matrix with elements $\lambda_i$ in the diagonal. It is not easy at this point to asses if such an endeavor would be worth pursuing, or for which sets of exponents \eqref{exponentes} would be feasible to study \eqref{nuevaz}. To offer some support for it, we would like to explore the particular case  
$$
N=3,\quad \alpha_1=\alpha_2=1/2, \alpha_3=0,
$$
so that equation \eqref{conformal} is replaced by 
\begin{equation}\label{particular}
\nabla\bu\,\nabla\bu^T=\begin{pmatrix}\det\nabla\bu&0&0\\0&\det\nabla\bu&0\\0&0&1\end{pmatrix},
\end{equation}
which corresponds to \eqref{nuevaz} for the indicated set of exponents. 

Let $\Omega\subset\R^3$ be a bounded, connected, Lipschitz domain, and take a unitary gradient
$$
 |\nabla w(\bbx)|^2=1\hbox{ a.e. }\bbx\in\Omega,
$$
for a Lipschitz function $w$. Notice that this condition implies that $w$ is basically a distance function. We use the standard wedge product for 3D vectors, determined through the identity 
$$
\bu\cdot(\bv\wedge\bw)=\det(\bu, \bv, \bw),\quad \bu, \bv, \bw\in\R^3.
$$
\begin{definition}
Two harmonic functions $u$ and $v$ in $\Omega$ are said to be conjugate of each other with respect to $w$, if
$$
\nabla u(\bbx)=\nabla v(\bbx)\wedge\nabla w(\bbx),\quad \nabla v(\bbx)=\nabla w(\bbx)\wedge\nabla u(\bbx),
$$
for a.e. $\bbx\in\Omega$.
\end{definition}
Note that the particular choice $w(x_1, x_2, x_3)=x_3$ takes us back to the classic 2D situation. It is elementary to argue that this definition corresponds precisely to \eqref{particular}. 

Our main result is the existence of non-trivial pairs of harmonic functions in 3D for every unitary gradient. 
\begin{theorem}\label{principal}
Under the conditions indicated for $\Omega$ and $w$, 
there are always non-trivial pairs of harmonic functions $(u, v)$ in $\Omega$ such that
\begin{equation}\label{identidadz}
\nabla u(\bbx)=\nabla v(\bbx)\wedge\nabla w(\bbx),\quad \nabla v(\bbx)=\nabla w(\bbx)\wedge\nabla u(\bbx)
\end{equation}
a. e. $\bbx$ in $\Omega$. 
\end{theorem}
From \eqref{identidadz} or \eqref{particular}, we immediately see that $\{\nabla u(\bbx), \nabla v(\bbx), \nabla w(\bbx)\}$ is a orthogonal basis of $\R^3$ for almost every $\bbx\in\Omega$, and moreover
$$
|\nabla u(\bbx)|^2=|\nabla v(\bbx)|^2=\det(\nabla u(\bbx), \nabla v(\bbx), \nabla w(\bbx)).
$$
Even more
$$
\nabla w(\bbx)=\frac1{|\nabla u(\bbx)|\,|\nabla v(\bbx)|}\nabla u(\bbx)\wedge\nabla v(\bbx)\hbox{ a.e. }\bbx\in\Omega.
$$
Being $u$ and $v$ harmonic, the product in the denominator must vanish in those places where $w$ is not smooth. 

More in general, we will also show the following. 
\begin{theorem}
Let $\Omega$ and $w$ be as before. 
Let $\gamma(\bbx)$ be a measurable function such that 
$$
0<C\le\gamma(\bbx)\le\frac1C\hbox{ in }\Omega.
$$
There are non-trivial pairs of functions $(u, v)$ in $H^1(\Omega)$ such that
$$
\gamma(\bbx)\nabla u(\bbx)=\nabla v(\bbx)\wedge\nabla w(\bbx),\quad \frac1{\gamma(\bbx)}\nabla v(\bbx)=\nabla w(\bbx)\wedge\nabla u(\bbx)
$$
a. e. $\bbx$ in $\Omega$. 
\end{theorem}

Even though the 2D case is something very well-established, we will explore it with some care (Section \ref{dos}) as a preliminary step,  from a variational perspective by looking at a certain vector variational problem to point the path for the generalization to the 3D case (Section \ref{tres}), and for some other situations examined in Sections \ref{cinco} and \ref{seis}. 

More specifically, we will treat the following issues.
\begin{itemize}
\item Relationship between boundary conditions for $u$ and $v$ coming from \eqref{identidadz} (Section \ref{cuatro}). This point will lead us to new types of boundary conditions that have been introduced in \cite{pedregal}. 
\item Changes for a non-unitary gradient (Section \ref{cinco}). Given that the restriction of a gradient being unitary may seem a bit restrictive, we explore what changes are introduced by being dispensed with such a point-wise constraint.
\item Potential relevance for Calderón's problem in 3D (Section \ref{seis}). We believe that these ideas may have some significance for the classic inverse problem in conductivity for the 3D situation. In fact, we conjecture that the fundamental uniqueness result in \cite{astala} for the 2D case will fail for $N=3$.
\end{itemize}

\section{The classical, 2D case}\label{dos}
The results in this section are classical. Our intention is to describe a path that may possibly 
allow the extension to the 3D situation. 

We would like to examine the vector, variational problem
\begin{equation}\label{dosde}
\hbox{Minimize in }(u, v)\in H^1(\Omega; \R^2):\quad \frac12\int_\Omega(|\nabla u(\bbx)|^2+|\nabla v(\bbx)|^2)\,d\bbx
\end{equation}
subject to
\begin{equation}\label{dosdei}
\int_\Omega \det(\nabla u(\bbx), \nabla v(\bbx))\,d\bbx=c(\neq0), \hbox{ a constant}. 
\end{equation}
Recall that
$$
\det(\bu, \bv)=-\bu\cdot\bQ\bv=\bQ\bu\cdot\bv,\quad \bu, \bv\in\R^2, \bQ=\begin{pmatrix}0&-1\\1&0\end{pmatrix}.
$$
We are not especially interested in any particular value of the constant $c$, as long as it is non-vanishing to discard trivial solutions. We claim that there are always, for every non-vanishing constant $c$, minimizers of this problem. Notice how the direct method stumbles with the fact that the $\det$ function is quadratic in 2D, exactly as the coercivity growth exponent given by the quadratic norm. Hence, despite the fact that $\det$ is weakly continuous, one cannot, at first sight, deduce the weak convergence
\begin{equation}\label{det2D}
\det(\nabla u_j, \nabla v_j)\rightharpoonup\det(\nabla u, \nabla v)\hbox{ in }L^1(\Omega)
\end{equation}
only under 
$$
(u_j, v_j)\rightharpoonup(u, v)\hbox{ in }H^1(\Omega; \R^2).
$$
\begin{proposition}
For every non-null constant $c$, there are minimizers $(u_c, v_c)\in H^1(\Omega; \R^2)$ for the variational problem \eqref{dosde}-\eqref{dosdei}. 
\end{proposition}
\begin{proof}
Let $(u_j, v_j)$ be a minimizing sequence. It is clear that we can assume that it is a uniformly bounded sequence in $H^1(\Omega; \R^2)$ by normalizing their averages over $\Omega$, and use the classic Poincaré-Wirtinger inequality to control the norm of functions. We have not done so explicitly in our variational problem for the sake of not introducing purely circumstantial data. As indicated above, there is apparently no way to ensure the weak convergence in \eqref{det2D} from a uniform bound in $H^1(\Omega; \R^2)$. 

For each $j$, consider the (scalar) auxiliary variational principle
$$
\hbox{Minimize in }u\in H^1(\Omega): \quad \int_\Omega\frac12|\nabla u(\bbx)|^2\,d\bbx
$$
subject to
$$
\int_\Omega \det(\nabla u(\bbx), \nabla v_j(\bbx))\,d\bbx=c.
$$
This time it is straightforward to conclude through the direct method that there is a unique (except for an additive constant that is normalize through the Poincaré-Wirtinger inequality as remarked earlier) minimizer $U_j\in H^1(\Omega)$. Since $u_j$ is feasible for this auxiliary problem, we will have 
$$
\frac12\int_\Omega|\nabla u_j(\bbx)|^2\,d\bbx\ge\frac12\int_\Omega|\nabla U_j(\bbx)|^2\,d\bbx,
$$
and hence $(U_j, v_j)$ will also be minimizing for our initial vector problem. 

On the other hand, by the optimality conditions that $U_j$ should comply with, we conclude that each $U_j$ must be harmonic in $\Omega$. Indeed, if we introduce a multiplier $\lambda$ to deal with the integral constraint, we would have that
$$
\dv(\nabla U_j+\lambda\bQ\nabla v_j)=0\hbox{ in }\Omega.
$$
The divergence of the second term identically vanishes being a rotated gradient in 2D. Thus we conclude the existence of a uniformly bounded (in $H^1(\Omega; \R^2)$), minimizing sequence  $(U_j, v_j)$ for our original vector problem such that each $U_j$ is harmonic in $\Omega$. If $(u, v)$ is a weak limit for some subsequence, it is a classic property that 
$$
U_j\to u, v_j\rightharpoonup v\hbox{ in }H^1(\Omega),
$$
the first strong convergence being a consequence of harmonicity. 
It is then clear that 
$$
\int_\Omega\det(\nabla U_j(\bbx), \nabla v_j(\bbx))\,d\bbx\to\int_\Omega\det(\nabla u(\bbx), \nabla v(\bbx))\,d\bbx,
$$
and $(u, v)$ becomes a true minimizer for our problem. 
\end{proof}

Our next step is to explore optimality conditions for these minimizers to conclude the following. 
\begin{proposition}
There are non-trivial pairs $(u, v)$ of harmonic functions in $\Omega$ such that
\begin{equation}\label{C-R}
\nabla u(\bbx)+\bQ\nabla v(\bbx)=\bcero\hbox{ in }\Omega.
\end{equation}
\end{proposition}
\begin{proof}
Let $(u_c, v_c)$ be a minimizer pair for each non-vanishing constant $c$, as a result of the previous proposition. 
Since we are not particularly interested in the specific value of the constant $c$ (as long as it does not vanish) as remarked earlier, and the family of minimizers $(u_c, v_c)$ still is so for a positive multiple of the $L^2$-norm of the gradient, we should have that $(u, v)(\equiv(u_c, v_c))$ must comply with optimality conditions for the functional
$$
\int_\Omega\left(\frac12|\nabla u(\bbx)|^2+\frac12|\nabla v(\bbx)|^2+\nabla u(\bbx)\cdot\bQ\nabla v(\bbx)\right)\,d\bbx, 
$$
under no boundary condition around $\partial\Omega$. Therefore, it holds that
\begin{gather}
\dv(\nabla u+\bQ\nabla v)=0\hbox{ in }\Omega,\quad (\nabla u+\bQ\nabla v)\cdot\bbn=0\hbox{ on }\partial\Omega,\nonumber\\
\dv(\nabla v-\bQ\nabla u)=0\hbox{ in }\Omega,\quad (\nabla v-\bQ\nabla u)\cdot\bbn=0\hbox{ on }\partial\Omega,\nonumber
\end{gather}
where $\bbn$ is the outer, unit normal to $\partial\Omega$. Of course, the differential equations in $\Omega$ imply that both $u$ and $v$ are harmonic in $\Omega$. But because of the natural boundary conditions on the right-hand sides,
$$
\int_\Omega(\nabla u(\bbx)+\bQ\nabla v(\bbx))\cdot\nabla U(\bbx)\,d\bbx=\int_\Omega(\nabla v(\bbx)-\bQ\nabla u(\bbx))\cdot\nabla U(\bbx)\,d\bbx=0
$$
for every $U\in H^1(\Omega)$. Or, alternatively, one performs variations in the augmented functional above with arbitrary functions $U$ in $H^1(\Omega)$ for both variables $u$ and $v$. In particular, 
\begin{gather}
\int_\Omega(\nabla u(\bbx)+\bQ\nabla v(\bbx))\cdot\nabla u(\bbx)\,d\bbx=0,\nonumber\\
\int_\Omega(\nabla v(\bbx)-\bQ\nabla u(\bbx))\cdot\nabla v(\bbx)\,d\bbx=0,\nonumber
\end{gather}
and, after some elementary algebra, we arrive at
$$
\int_\Omega \frac12|\nabla u(\bbx)+\bQ\nabla v(\bbx)|^2\,d\bbx=0.
$$
\end{proof}
Note that it would be hard to achieve the existence of non-trivial pairs $(u, v)$ of solutions to \eqref{C-R} by working directly with the functional
$$
\int_\Omega |\nabla u(\bbx)+\bQ\nabla v(\bbx)|^2\,d\bbx.
$$

\section{The 3D case}\label{tres}
Suppose $\Omega\subset\R^3$ is a model domain (like a ball, a cube, or a cylinder), and $w(\bbx)$ is a given Lipschitz function with unitary gradient
\begin{equation}\label{unidad}
|\nabla w(\bbx)|^2=1\hbox{ a.e. }\bbx\in\Omega. 
\end{equation}

We would like to investigate the vector variational problem
\begin{equation}\label{tresdez}
\hbox{Minimize in }(u, v)\in H^1(\Omega; \R^2):\quad \frac12\int_\Omega(|\nabla u(\bbx)|^2+|\nabla v(\bbx)|^2)\,d\bbx
\end{equation}
under the integral constraint
\begin{equation}\label{tresdeiz}
\int_\Omega \nabla u(\bbx)\cdot(\nabla v(\bbx)\wedge\nabla w(\bbx))\,d\bbx=c(\neq0), \hbox{ a constant}.
\end{equation}
Note that
$$
\det(\nabla u, \nabla v, \nabla w)=\nabla u\cdot(\nabla v\wedge\nabla w).
$$
We restate our main result.
\begin{theorem}\label{teor3d}
Let the function $w(\bbx)$ be given with the indicated properties above. There are non-trivial pairs of harmonic functions $(u, v)$ in $\Omega$ such that
\begin{equation}\label{identidad}
\nabla u(\bbx)=\nabla v(\bbx)\wedge\nabla w(\bbx),\quad \nabla v(\bbx)=\nabla w(\bbx)\wedge\nabla u(\bbx)
\end{equation}
a. e. $\bbx$ in $\Omega$. 
\end{theorem}
The proof of our theorem follows along the lines of the 2D case in Section \ref{dos}. At this point, the proof carefully treads over that path. We distinguish two steps.

Step 1. Existence of minimizers for \eqref{tresdez}-\eqref{tresdeiz} (and a normalizing condition on the average of functions over $\Omega$ which we omit). The proof is exactly the same as with the 2D case. If $(u_j, v_j)$ is a minimizing sequence converging weakly in $H^1(\Omega; \R^2)$ to $(u, v)$, we focus on the scalar variational problem
$$
\hbox{Minimize in }u\in H^1(\Omega):\quad \frac12\int_\Omega|\nabla u(\bbx)|^2\,d\bbx
$$
under
$$
\int_\Omega \nabla u(\bbx)\cdot(\nabla v_j(\bbx)\wedge\nabla w(\bbx))\,d\bbx=c.
$$
As before, there is a minimizer $U_j$ which must be harmonic because, by the classical Piola identity, 
$$
\dv(\nabla v_j\wedge\nabla w)=0\hbox{ in }\Omega.
$$
Since $u_j$ is feasible for this last problem, it turns out that $(U_j, v_j)$ is also minimizing for the initial variational problem, and hence uniformly bounded in $H^1(\Omega; \R^2)$. If the pair $(u, v)$ is a weak limit of some suitable subsequence, since each $U_j$ is harmonic, then
\begin{equation}\label{convergenciafd}
U_j\to u, v_j\rightharpoonup v\hbox{ in }H^1(\Omega),
\end{equation}
 $(u, v)$ becomes feasible because
 $$
 \int_\Omega \nabla U_j(\bbx)\cdot(\nabla v_j(\bbx)\wedge\nabla w(\bbx))\,d\bbx\to
 \int_\Omega \nabla u(\bbx)\cdot(\nabla v(\bbx)\wedge\nabla w(\bbx))\,d\bbx
 $$
 under \eqref{convergenciafd} and the uniform convergence of the gradient of $w$, and hence a minimizer for our problem.
 
 Step 2. We explore optimality. Arguing as in the 2D case, we can conclude that the minimizer found in Step 1, ought to comply with optimality conditions for the augmented functional
 $$
 \int_\Omega\left(\frac12|\nabla u(\bbx)|^2+|\nabla v(\bbx)|^2-\nabla u(\bbx)\cdot(\nabla v(\bbx)\wedge\nabla w(\bbx)\right)\,d\bbx,
 $$
under no boundary restriction around $\partial\Omega$ whatsoever. Therefore, we must have
\begin{gather}
\int_\Omega\left(\nabla u(\bbx)-\nabla v(\bbx)\wedge\nabla w(\bbx)\right)\cdot\nabla U(\bbx)\,d\bbx=0,\nonumber\\
\int_\Omega\left(\nabla v(\bbx)-\nabla w(\bbx)\wedge\nabla u(\bbx)\right)\cdot\nabla U(\bbx)\,d\bbx=0,\nonumber
\end{gather}
for every $U\in H^1(\Omega)$. In particular
\begin{gather}
\int_\Omega\left(\nabla u(\bbx)-\nabla v(\bbx)\wedge\nabla w(\bbx)\right)\cdot\nabla u(\bbx)\,d\bbx=0,\nonumber\\
\int_\Omega\left(\nabla v(\bbx)-\nabla w(\bbx)\wedge\nabla u(\bbx)\right)\cdot\nabla v(\bbx)\,d\bbx=0.\nonumber
\end{gather}
From these two identities, we conclude that
$$
\int_\Omega\left(\frac12|\nabla u(\bbx)|^2+|\nabla v(\bbx)|^2-\nabla u(\bbx)\cdot(\nabla v(\bbx)\wedge\nabla w(\bbx)\right)\,d\bbx=0.
$$
Taking advantage of the fact that $\nabla w$ is a unitary gradient, we can rewrite the previous equality in the form
\begin{equation}\label{identidad}
\int_\Omega\left(\frac12|\nabla u(\bbx)-\nabla v(\bbx)\wedge\nabla w(\bbx)|^2+\frac12(\nabla v(\bbx)\cdot\nabla w(\bbx))^2\right)\,d\bbx=0.
\end{equation}
Recall that
$$
|\bv\wedge\bw|^2+(\bv\cdot\bw)^2=|\bv|^2\,|\bw|^2,
$$
for three dimensional vectors $\bv, \bw\in\R^3$. From \eqref{identidad}, we conclude that
\begin{equation}\label{consecuencia}
\nabla u(\bbx)=\nabla v(\bbx)\wedge\nabla w(\bbx),\quad \nabla v(\bbx)\cdot\nabla w(\bbx)=0
\end{equation}
for a.e. $\bbx\in\Omega$. Based on \eqref{consecuencia}, we compute
$$
\nabla w(\bbx)\wedge\nabla u(\bbx)=\nabla w(\bbx)\wedge(\nabla v(\bbx)\wedge\nabla w(\bbx)),
$$
and taking into account the formula
\begin{equation}\label{igualdady}
\bw\wedge(\bv\wedge\bu)=\bw\cdot\bu\,\bv-\bv\cdot\bu\,\bw
\end{equation}
again valid for arbitrary vectors $\bu, \bv, \bw\in\R^3$, we arrive at
$$
\nabla v(\bbx)=\nabla w(\bbx)\wedge\nabla u(\bbx) \hbox{ for a.e. }\bbx\in\Omega, 
$$
thanks to the second part of \eqref{consecuencia}, and \eqref{unidad}.

\section{Boundary conditions and uniqueness of representation}\label{cuatro}
Just as in the 2D case, where the central vector equation
$$
\nabla u(\bbx)+\bQ\nabla v(\bbx)=\bcero
$$
expressing the intimate link between $u$ and $v$, can be exploited to find one of the two from the other, one can try the same strategy in the 3D case.

\begin{proposition}\label{una-otra}
Let $\Omega\subset\R^3$ be as before, and let $v(\bbx)$ be harmonic in $\Omega$. Then the unique solution $u$ of the scalar variational problem
$$
\hbox{Minimize in }u\in H^1(\Omega):\quad \int_\Omega\left(\frac12|\nabla u(\bbx)|^2+\nabla u(\bbx)\cdot(\nabla v(\bbx)\wedge\nabla w(\bbx))\right)\,d\bbx
$$
is the conjugate harmonic function to $v$ with respect to $w$.
\end{proposition}

If we formally define the unit tangent vector field to $\partial\Omega$ given by
\begin{equation}\label{tangente}
\bbt(\bbx)=\frac1{\phi(\bbx)}\nabla w(\bbx)\wedge\bbn(\bbx),\quad \phi(\bbx)=|\nabla w(\bbx)\wedge\bbn(\bbx)|,
\end{equation}
then we can also find $u$ from $v$  by formally looking at the alternative variational problem
$$
\hbox{Minimize in }u\in H^1(\Omega):\quad \int_\Omega\frac12|\nabla u(\bbx)|^2\,d\bbx+\int_{\partial\Omega} u(\bbx) \nabla v(\bbx)\cdot\bbt(\bbx) \phi(\bbx)\,dS(\bbx).
$$
In fact, the boundary condition coming from optimality for the variational problem in Proposition \ref{una-otra} amounts to
$$
(\nabla u+\nabla v\wedge\nabla w)\cdot\bbn=0\hbox{ on }\partial\Omega.
$$
Yet, apparently, it is not possible to interpret the isolated condition 
\begin{equation}\label{nueva}
(\nabla v\wedge\nabla w)\cdot\bbn=0\hbox{ on }\partial\Omega
\end{equation}
with a proper precise meaning for the function $v$, beyond formally saying that the tangential derivative of $v$ along the tangent vector field $\bbt$ in \eqref{tangente} should vanish. This step looks important when exploring the possibilities of using this approach to tackle Calderón's problem in 3D. As a matter of fact, such boundary conditions have been introduced and explored in \cite{pedregal} precisely motivated by the standard inverse conductivity problem in 3D. Roughly speaking, the subspace of functions of $H^1(\Omega)$ for which \eqref{nueva} is valid is 
$$
\bbH_w=H^1_0(\Omega)+\bbL_w, \quad \bbL_w=\{\phi(w)\in H^1(\Omega)\}.
$$
Note that, in cases of interest, 
$$
H^1_0(\Omega)\subsetneq \bbH_w\subsetneq H^1(\Omega),
$$
and so, variational problems posed on these subspaces lead to boundary conditions properly between Dirichlet and Neumann (in a very different way compared to mixed problems). All of these issues are investigated in detail in \cite{pedregal}. 

It is interesting to notice how functions $v$ in the subspace $\bbL_w$ are such that
$$
\nabla v\wedge\nabla w=\bcero\hbox{ in }\Omega.
$$
This is related to the fact that pairs of conjugate harmonic functions with respect to $w$ can never be functions of $w$ itself. This is clearly seen in constraint \eqref{tresdeiz}, or in the variational problem in Proposition \ref{una-otra}. Other than this comment, pairs of conjugate harmonic functions in 3D with respect to a given unitary gradient $|\nabla w|^2=1$ are determined in a unique way except for arbitrary additive constants. 

\section{The non-unitary situation}\label{cinco}
These are three explicit examples, valid if the domain $\Omega$ is taken to be the unit ball $\bB$ of $\R^3$, 
\begin{gather}
w(x_1, x_2, x_3)=\sqrt{x_1^2+x_2^2+(x_3+2)^2},\nonumber\\
w(x_1, x_2, x_3)=\sqrt{x_1^2+(x_3+2)^2},\nonumber\\
w(x_1, x_2, x_3)=x_3+2.\nonumber
\end{gather}
Taking supremum or infimum on functions of this nature, more explicit examples can be written. However, the unitary condition on a gradient looks like a somewhat tedious constraint to deal with. What is lost in our main result Theorem \ref{principal} if we simply consider a function $w$ with no condition on the size of its gradient?

If we insist in keeping essentially the same two main steps of the proof of Theorem \ref{teor3d}, we loose the harmonicity of one of the two functions, but the other remains harmonic. 

\begin{theorem}\label{armnoarm}
Let $w(\bbx):\Omega\subset\R^3\to\R^3$ be a Lipschitz function with
$$
0<C\le |\nabla w(\bbx)|^2\le \frac1C\hbox{ in }\Omega, \quad C>0.
$$
Then there is a harmonic function $u$, and a solution $v$ of the equation
$$
\dv(|\nabla w(\bbx)|^2\nabla v)=0\hbox{ in }\Omega,
$$
such that
$$
\nabla u(\bbx)=\nabla v(\bbx)\wedge\nabla w(\bbx),\quad |\nabla w(\bbx)|^2\nabla v(\bbx)=\nabla w(\bbx)\wedge\nabla u(\bbx)
$$
a. e. $\bbx$ in $\Omega$. 
\end{theorem}
As claimed, the proof follows line by line that of Theorem \ref{teor3d} under suitable, minor changes. We explore the vector variational problem
\begin{equation}\label{tresde}
\hbox{Minimize in }(u, v)\in H^1(\Omega; \R^2):\quad \frac12\int_\Omega(|\nabla u(\bbx)|^2+|\nabla w(\bbx)|^2\,|\nabla v(\bbx)|^2)\,d\bbx
\end{equation}
under the integral constraint
\begin{equation}\label{tresdei}
\int_\Omega \nabla u(\bbx)\cdot(\nabla v(\bbx)\wedge\nabla w(\bbx))\,d\bbx=c(\neq0), \hbox{ a constant}.
\end{equation}
Step 1 of the proof of Theorem \ref{teor3d} runs exactly in the same way to conclude the existence of minimizers $(u, v)$ for this constrained variational problem. Note that it looks like a fundamental assumption to rely on the harmonicity of $U_j$ for a uniformly bounded, minimizing sequence $(U_j, v_j)$. Moving to Step 2, we conclude that such minimizers should comply with optimality conditions for the augmented functional
$$
 \int_\Omega\left(\frac12|\nabla u(\bbx)|^2+|\nabla w(\bbx)|^2\,|\nabla v(\bbx)|^2-\nabla u(\bbx)\cdot(\nabla v(\bbx)\wedge\nabla w(\bbx)\right)\,d\bbx,
 $$
 namely, 
\begin{gather}
\int_\Omega\left(\nabla u(\bbx)-\nabla v(\bbx)\wedge\nabla w(\bbx)\right)\cdot\nabla U(\bbx)\,d\bbx=0,\nonumber\\
\int_\Omega\left(|\nabla w(\bbx)|^2\,\nabla v(\bbx)-\nabla w(\bbx)\wedge\nabla u(\bbx)\right)\cdot\nabla U(\bbx)\,d\bbx=0,\nonumber
\end{gather}
for every $U\in H^1(\Omega)$. In particular
\begin{gather}
\int_\Omega\left(\nabla u(\bbx)-\nabla v(\bbx)\wedge\nabla w(\bbx)\right)\cdot\nabla u(\bbx)\,d\bbx=0,\nonumber\\
\int_\Omega\left(|\nabla w(\bbx)|^2\,\nabla v(\bbx)-\nabla w(\bbx)\wedge\nabla u(\bbx)\right)\cdot\nabla v(\bbx)\,d\bbx=0.\nonumber
\end{gather}
From these two identities, we conclude that
$$
\int_\Omega\left(\frac12|\nabla u(\bbx)|^2+|\nabla w(\bbx)|^2\,|\nabla v(\bbx)|^2-\nabla u(\bbx)\cdot(\nabla v(\bbx)\wedge\nabla w(\bbx)\right)\,d\bbx=0.
$$
The integral on the left-hand side can be rewritten in the form
$$
\int_\Omega\left(\frac12|\nabla u(\bbx)-\nabla v(\bbx)\wedge\nabla w(\bbx)|^2+\frac12(\nabla v(\bbx)\cdot\nabla w(\bbx))^2\right)\,d\bbx=0,
$$
in such a way that
$$
\nabla u(\bbx)=\nabla v(\bbx)\wedge\nabla w(\bbx),\quad \nabla v(\bbx)\cdot\nabla w(\bbx)=0,
$$
for a.e. $\bbx\in\Omega$. By utilizing \eqref{igualdady}, we find
$$
\nabla w(\bbx)\wedge\nabla u(\bbx)=|\nabla w(\bbx)|^2 \nabla v(\bbx).
$$
Note that in this case
$$
\det(\nabla u(\bbx), \nabla v(\bbx), \nabla w(\bbx))=|\nabla u(\bbx)|^2=|\nabla v(\bbx)|^2\,|\nabla w(\bbx)|^2.
$$

\section{A broader context}\label{seis}
The analysis in Sections \ref{dos} and \ref{tres} can be generalized without much effort to a more general setting closer to the classical inverse problem in conductivity. Suppose the conductivity coefficient $\gamma$ is measurable and 
\begin{equation}\label{coefcond}
0<C\le\gamma(\bbx)\le\frac1C\hbox{ in }\Omega, \hbox{ for a.e. }\bbx\in\Omega.
\end{equation}
\begin{proposition}
There are non-trivial pairs $(u, v)$ of functions in $H^1(\Omega)$ such that
$$
\gamma(\bbx)\nabla u(\bbx)+\bQ\nabla v(\bbx)=\bcero\hbox{ in }\Omega.
$$
\end{proposition}
As it is standard, one readily sees that $u$ and $v$ are solutions of the respective conductivity equations
$$
\dv(\gamma\nabla u)=0,\quad \dv\left(\frac1\gamma\nabla v\right)=0
$$
in $\Omega$, and there is an intimate connection between boundary values of both around $\partial\Omega$. 
\begin{proof}
The proof is formally the same as that in Section \ref{dos}. The starting point is the modified functional
$$
\frac12\int_\Omega\left(\gamma(\bbx)|\nabla u(\bbx)|^2+\frac1{\gamma(\bbx)}|\nabla v(\bbx)|^2\right)\,d\bbx
$$
under exactly the same integral constraint \eqref{dosdei}. The proof proceeds in the same terms as in Section \ref{dos}. The only point that deserves an important comment relates to the strong convergence $\nabla U_j\to\nabla u$ in $L^2(\Omega)$ that must be deduced from the information  
$$
\nabla U_j\rightharpoonup\nabla u\hbox{ in }L^2(\Omega),\quad \dv(\gamma\nabla U_j)=0\hbox{ in }\Omega.
$$
Recall that in the case $\gamma\equiv1$, that fact was achieved by exploiting the harmonicity of the modified first component $u_j\mapsto U_j$ of a minimizing sequence of pairs $(u_j, v_j)$. In this more general framework, we invoke the classical div-curl lemma (\cite{murat}, \cite{murat2}, \cite{tartar}). Indeed, as a result of this fundamental principle and the above differential equation that $U_j$ complies with, one can conclude that
\begin{equation}\label{integrales}
\int_\Omega \gamma(\bbx)\nabla U_j(\bbx)\cdot\nabla U_j(\bbx)\,d\bbx\to \int_\Omega \gamma(\bbx)\nabla u(\bbx)\cdot\nabla u(\bbx)\,d\bbx.
\end{equation}
The strict convexity of the quadratic functional
$$
U\in H^1(\Omega)\mapsto\int_\Omega\gamma(\bbx)|\nabla U(\bbx)|^2\,d\bbx
$$
together with the weak convergence 
$$
\nabla U_j\rightharpoonup\nabla u\hbox{ in }L^2(\Omega)
$$
and the convergence of the integrals in \eqref{integrales} imply the desired strong convergence 
\begin{equation}\label{fuerte}
\nabla U_j\to\nabla u\hbox{ in }L^2(\Omega).
\end{equation}
This is something well-established.  Though it will most likely be written  in several places, one can check  the version in Theorem 3.16 in \cite{pedregalb} for a similar result in a more general setting.
Once the strong convergence in \eqref{fuerte} is shown, the rest of the proof follows line by line the previous one with the corresponding changes. 
\end{proof}
The extension to the 3D situation, after the remark in the preceding proof, is now straightforward with the obvious changes. Suppose $\Omega\subset\R^3$ is a domain as in Section \ref{tres}. Let $\gamma$ verify \eqref{coefcond}, and let $w(\bbx)$ be a Lipschitz function with a unitary gradient $|\nabla w(\bbx)|^2=1$ for a.e. $\bbx\in\Omega$. 
\begin{theorem}\label{final}
There are non-trivial pairs of functions $(u, v)$ in $H^1(\Omega)$ such that
\begin{equation}\label{identidadg}
\gamma(\bbx)\nabla u(\bbx)=\nabla v(\bbx)\wedge\nabla w(\bbx),\quad \frac1{\gamma(\bbx)}\nabla v(\bbx)=\nabla w(\bbx)\wedge\nabla u(\bbx)
\end{equation}
a. e. $\bbx$ in $\Omega$. 
\end{theorem}
The point we would like to stress is that given a solution $u$ of the conductivity equation
$$
\dv(\gamma \nabla u)=0\hbox{ in }\Omega,
$$
there are infinitely-many $v'$s, one for each feasible $w$, for which vector equation \eqref{identidadg} is valid. 

These ideas are reminiscent of the classical Calderón's inverse conductivity problem. Let $\Omega$ be a bounded, simply-connected domain in $\R^N$, $N=2, 3$. 
Given a valid conductivity coefficient $\gamma$ as in \eqref{coefcond}, we define the Dirichlet-to-Neumann operator
$$
\Lambda_\gamma:H^{1/2}(\Omega)\mapsto H^{-1/2}(\Omega)
$$
determined by putting
$$
\Lambda_\gamma(u^\circ)=v^\circ,\quad v^\circ=\left.\frac{\partial u}{\partial\bbn}\right|_{\partial\Omega},
$$
where $u\in H^1(\Omega)$ is the unique solution of the problem
$$
\dv(\gamma \nabla u)=0\hbox{ in }\Omega,\quad u=u^\circ\hbox{ on }\partial\Omega.
$$
In \cite{astala}, it was shown that for $N=2$, the map $\gamma\mapsto\Lambda_\gamma$ is one-to-one. 
It has been conjectured that this is not so for higher dimension (see \cite{railoz} for a recent account). It is somewhat natural to support this conjecture from our point of view here since  the presence of the undetermined unit gradient $|\nabla w(\bbx)|^2=1$ might be somehow responsible for the lack of uniqueness of the operator $\Lambda$ with respect to the conductivity coefficient it comes from. A fundamental issue in this area is the regularity where conductivity coefficients are searched for since uniqueness results or counterexamples may dramatically depend on such regularity assumed on classes of conductivity coefficients (check again \cite{railoz}). 

%
In the same vein as Theorem \ref{final}, exploiting the use of the conductivity coefficient $\gamma$ and the norm of the gradient $|\nabla w|$, one can easily show the following variant of Theorem \ref{armnoarm}. 
\begin{corollary}
Let $w(\bbx):\Omega\subset\R^3\to\R^3$ be a Lipschitz function with
$$
0<C\le |\nabla w(\bbx)|^2\le \frac1C\hbox{ in }\Omega, \quad C>0.
$$
Then there are two solutions $u, v\in H^1(\Omega)$ of the conductivity equation
$$
\dv(|\nabla w(\bbx)|\nabla U)=0\hbox{ in }\Omega,\quad U=u, v, 
$$
such that
$$
|\nabla w(\bbx)|\nabla u(\bbx)=\nabla v(\bbx)\wedge\nabla w(\bbx),\quad |\nabla w(\bbx)|\nabla v(\bbx)=\nabla w(\bbx)\wedge\nabla u(\bbx)
$$
a. e. $\bbx$ in $\Omega$. 
\end{corollary}


\begin{thebibliography}{99}
\bibitem{alhfors} Lectures on quasi-conformal mappings
Ahlfors, Lars V., Van Nostrand Mathematical Studies, No. 10
D. Van Nostrand Co., Inc., Toronto, Ont.-New York-London, 1966, v+146 pp.
\bibitem{astala}  Calderón's inverse conductivity problem in the plane
Astala, K.; Päivärinta, L., Ann. of Math. (2) 163 (2006), no. 1, 265–299.
\bibitem{astaliwa}  Elliptic partial differential equations and quasi-conformal mappings in the plane
Astala, K.; Iwaniec, T.; Martin, G., Princeton Math. Ser., 48
Princeton University Press, Princeton, NJ, 2009, xviii+677 pp.
\bibitem{iwaniec} Quasiregular mappings in even dimensions
Iwaniec, T.; Martin, G., 
Acta Math. 170 (1993), no. 1, 29–81.
\bibitem{lehto}  Quasiconformal mappings in the plane
Lehto, O.; Virtanen, K. I., Die Grundlehren der mathematischen Wissenschaften, Band 126
Springer-Verlag, New York-Heidelberg, 1973, viii+258 pp.
\bibitem{murat} Murat, F., Compacité par compensation. (French) Ann. Scuola Norm. Sup. Pisa Cl. Sci. (4) 5 (1978), no. 3, 489–507.
\bibitem{murat2} A survey on compensated compactness
Murat, F., Pitman Res. Notes Math. Ser., 148
Longman Scientific \& Technical, Harlow, 1987, 145–183.
\bibitem{pedregalb} Parametrized measures and variational principles
Pedregal, P., Progr. Nonlinear Differential Equations Appl., 30
Birkhäuser Verlag, Basel, 1997, xii+212 pp.
\bibitem{pedregal}  On a new type of boundary condition,
Pedregal, P., Rev. R. Acad. Cienc. Exactas Fís. Nat. Ser. A Mat. RACSAM 116 (2022), no. 1, Paper No. 43, 14 pp.
\bibitem{railoz} Railo, J., Zimmermann, Ph., 
Low regularity theory for the inverse fractional conductivity problem,
Nonlinear Anal. 239 (2024), Paper No. 113418, 27 pp.
\bibitem{tartar} Compensated compactness and applications to partial differential equations
Tartar, L., Res. Notes in Math., 39
Pitman (Advanced Publishing Program), Boston, Mass.-London, 1979, pp. 136–212.
\end{thebibliography}
\end{document}